\documentclass[11pt]{article}
\usepackage{amsfonts, amsmath, amsthm, amssymb, latexsym, amscd, color}
\usepackage{tabularx,ragged2e,booktabs,caption}
\usepackage{rotating}

\usepackage{mathrsfs,calrsfs}
\usepackage{textcomp}
\usepackage{graphicx,tikz}
\newcommand{\comment}[1]{}
\newif\ifpdf
\ifpdf \else \fi \textwidth = 6.5 in \textheight = 9 in
\newtheorem{thm}{Theorem}[section]

\newtheorem{observation}[thm]{Observation}

\newtheorem{corollary}[thm]{Corollary}
\newtheorem{lemma}[thm]{Lemma}

\newtheorem{theorem}[thm]{Theorem}

\begin{document}

\title{Domination Critical Kn\"odel graphs}
\author{{\small D.A. Mojdeh}$^{a}$, {\small S.R. Musawi}$^{b}$ and {\small E. Nazari}$^{b}$ \\
{$^{a}$\small Department of
Mathematics, University of Mazandaran, Babolsar, Iran}\\
{\small Email: damojdeh@yahoo.com}\\
{$^{b}$\small Department of Mathematics, University of Tafresh, Tafresh, Iran}\\
{\small Email: r\_musawi@yahoo.com and nazari.esmaeil@gmail.com}\\
}
\date{}
\maketitle

\begin{abstract}
A set $D$ of vertices of a graph $G$ is a dominating set if each vertex of $V(G) \setminus D$, is adjacent to some vertex of $D$. The domination number of $G$, $\gamma(G)$, is the minimum cardinality of a dominating set of $G$. A graph $G$ is called domination vertex critical, or just $\gamma$-critical if removal of any vertex decreases the domination number. A graph $G$ is called domination vertex stable, or just $\gamma$-stable, if removal of any vertex does not decrease the domination number. For an even integer $n\ge 2$ and $1\le \Delta \le \lfloor\log_2n\rfloor$, a \textit{Kn\"odel graph} $W_{\Delta,n}$ is a $\Delta$-regular bipartite graph of even order $n$, with vertices $(i,j)$, for $i=1,2$ and $0\le j\le n/2-1$, where for every $j$, $0\le j\le n/2-1$, there is an edge between vertex $(1, j)$ and every vertex $(2,(j+2^k-1)$ mod (n/2)), for $k=0,1,\cdots,\Delta-1$. In this paper, we study the domination criticality and domination stability of Kn\"odel graphs. We characterize the  3-regular and 4-regular Kn\"odel graphs by $\gamma$-criticality or $\gamma$-stability.\\
\textbf{Keywords:} Kn\"odel graph, domination number, domination vertex critical graph.\\
\textbf{Mathematics Subject Classification [2010]:} 05C69, 05C30
\end{abstract}

\section{Introduction}
Let $G=(V,E)$ denote a simple graph of order $n=|V(G)|$ and size $m=|E(G)|$. Two vertices $u,v\in V(G)$ are \textit{adjacent} or \textit{neighbor} if $uv\in E(G)$. The \textit{open neighborhood} of a vertex $u\in V(G)$ is denoted by $N(u)=\{v\in V(G) | uv\in E(G)\}$ and for a vertex set $S\subseteq V(G)$, $N(S)=\underset{u\in S}{\cup}N(u)$. The cardinality of $N(u)$ is called the \textit{degree} of $u$ and is denoted by $\deg(u)$, (or $\deg_G(u)$ to refer it to $G$). The \textit{maximum degree} and \textit{minimum degree} among all vertices in $G$ are denoted by $\Delta(G)$ and $\delta(G)$, respectively. A graph $G$ is called \textit{$k$-regular} if $k=\delta(G)=\Delta(G)$. If $D$ is a set of vertices of a graph $G$, then a vertex $u\in V(G)\setminus D$ is a $D$-external private neighbor ($D$-epn) of the vertex $v\in D$ if $N(u)\cap D=\{v\}$. The set of all $D$-epns of $v$ is denoted by $epn(v,D)$. A graph $G$ is a \textit{bipartite graph} if its vertex set can partition to two disjoint sets $X$ and $Y$ such that each edge in $E(G)$ connects a vertex in $X$ with a vertex in $Y$.

A set $D\subseteq V(G)$ is a \textit{dominating set} if for each $u\in V(G) \setminus D$, $u$ is adjacent to some vertex $v\in D$. The \textit{domination number} of $G$, $\gamma(G)$, is the minimum cardinality among all dominating sets of $G$. For more on dominating set, interested reader may see \cite{hhs}. A vertex $u$ is a critical vertex if $\gamma(G-u)<\gamma(G)$ and a graph $G$ is called \textit{domination vertex critical} if for any vertex $u\in V(G)$, we have $\gamma(G-u)<\gamma(G)$. We say $G$ is a  $\gamma$-critical graph. A graph $G$ is called \textit{domination vertex stable} or \textit{$\gamma$-stable} if for any vertex $u\in V(G)$, we have $\gamma(G-u)=\gamma(G)$. One of the most important problems in domination theory is to determine graphs in which every vertex is critical or stable. Many authors have already studied the criticality and stability of several domination parameters, see for example \cite{bhns,bcd,cjk,hj,jsk,j,mf1,mf2, mh, s,sb}.

An automorphism of the graph $G$ is a permutation $\sigma$ on $V(G)$ such that the pair of vertices $(u,v)$ are adjacent if and only if the pair $(\sigma(u),\sigma(v))$ also are adjacent. A graph $G$ is called \textit{vertex transitive} if for each pair of vertices $u$ and $v$ in $V(G)$, there exists an automorphism $\sigma$ such that $\sigma(u)=v$. For other graph theory notation and terminology not given here, we refer to \cite{bm}.

An interesting family of graphs namely \textit{Kn\"odel graphs} have been introduced about 1975 \cite{k}, and have been studied seriously by some authors since 2001, see, for example, \cite{fr}. For an even integer $n\ge 2$ and $1\le \Delta \le \lfloor\log_2n\rfloor$, a \textit{Kn\"odel graph} $W_{\Delta,n}$ is a $\Delta$-regular bipartite graph of even order $n$, with vertices $(i,j)$, for $i=1,2$ and $0\le j\le n/2-1$, where for every $j$, $0\le j\le n/2-1$, there is an edge between vertex $(1, j)$ and every vertex $(2,(j+2^k-1)$ mod (n/2)), for $k=0,1,\cdots,\Delta-1$ (see \cite{xxyf}). Kn\"odel graphs, $W_{\Delta,n}$, are one of the three important families of graphs that they have good properties in terms of broadcasting and gossiping, see for example \cite{gh}. It is worth-noting that any Kn\"odel graph is a Cayley graph and so it is a vertex transitive graph, see \cite{fr}.

Xueliang et al. \cite{xxyf} obtained exact domination number for $W_{3,n}$ and Mojdeh et al. \cite{mmnj} obtained exact domination number for $W_{4,n}$ \cite{mmn}.

\begin{theorem}[Xueliang et. al. \cite{xxyf}]\label{t11}
For each integer $n\ge8$, we have \[\gamma(W_{3,n})=2\lfloor\frac{n}{8}\rfloor+\left\{\begin{array}{cc}0&n\equiv0 \emph{ (mod 8)}\\1&n\equiv2 \emph{ (mod 8)} \\2&n\equiv4,6 \emph{ (mod 8)}\end{array}  \right. .\]
\end{theorem}

\begin{theorem}[Mojdeh et al. \cite{mmn}]\label{t12}
For each integer $n\ge 16$, we have \[\gamma (W_{4,n})=2\lfloor\frac{n}{10}\rfloor+\left\{\begin{array}{cc}0&n\equiv0 \emph{ (mod 10)}\\2&n=16,18,36\,;\,n\equiv2,4 \emph{ (mod 10)} \\3&n=28\,;\,n\equiv6 \emph{ (mod 10)},n\ne16,36\\4&n\equiv8 \emph{ (mod 10)},n\ne18,28\end{array}  \right. .\]
\end{theorem}

In this paper, we study the domination criticality and domination stability of 3-regular and 4-regular Kn\"odel graphs. In section 3, we determine which 3-regular Kn\"odel graphs are $\gamma$-critical or $\gamma$-stable. In section 4, we determine which 4-regular Kn\"odel graphs are $\gamma$-critical or $\gamma$-stable.

We will use the following.

\begin{theorem}\label{t13}\cite{b,was}
If $G$ is a graph with $n$ vertices, then $\gamma(G)\ge\frac{n}{\Delta+1}$.
\end{theorem}

The following is straightforward.
\begin{observation}\label{o14}
If $G$ is a $\gamma$-critical graph, then $\gamma(G-u)+1=\gamma(G)$ for any vertex $u\in V(G)$.
\end{observation}


\section{\bf Preliminary}
For simplicity, we use from the re-labeling of the vertices of a Kn\"odel graph that states by Mojdeh et. al. in \cite{mmnj}, as follows: we label the vertex $(1,i)$ by $u_{i+1}$ for each $i=0,1,...,n/2-1$, and the vertex $(2,j)$ by $v_{j+1}$ for $j=0,1,...,n/2-1$. Let $U=\{u_1,u_2,\cdots, u_{\frac{n}{2}}\}$ and $V=\{v_1,v_2,\cdots, v_{\frac{n}{2}}\}$. From now on, the vertex set of each Kn\"odel graph $W_{\Delta,n}$ is $U\cup V$ such that $U$ and $V$ are the two partite sets of the graph. Note that two vertices $u_i$ and $v_j$ are adjacent if and only if $j\in \{i+2^0-1,i+2^1-1,\cdots,i+2^{\Delta-1}-1\}$, where the addition is taken in modulo $n/2$. If $S$ is a set of vertices of $W_{\Delta,n}$, then clearly, $S_U=S\cap U$ and $S_V=S\cap V$ partition $S$, $|S|=|S_U|+|S_V|$, $N(S_U)\subseteq V$ and $N(S_V)\subseteq U$. If $D$ is a dominating set of $W_{\Delta,n}$, then, $U\setminus N(D_V)\subseteq D_U$ and $V\setminus N(D_U)\subseteq D_V$. In \cite{mmnj}, for any subset $\{u_{i_1},u_{i_2},\cdots,u_{i_k}\}$ of $U$ with $1\le i_1 <i_2<\cdots< i_k\le\frac{n}{2}$, authors corresponded a sequence based on the differences of the indices of $u_j$, $j=i_1,...,i_k$, as follows and they introduced some results of this sequence.

\begin{def}\cite{mmnj}
For any subset $A=\{u_{i_1},u_{i_2},\cdots,u_{i_k}\}$ of $U$ with $1\le i_1 <i_2<\cdots< i_k\le\frac{n}{2}$ we define a sequence $n_1, n_2, \cdots, n_k$, namely \textbf{cyclic-sequence}, where $n_j=i_{j+1}-i_j$ for $1\le j\le k-1$ and $n_k=\frac{n}{2}+i_1-i_k$. For two vertices $u_{i_j}, u_{i_{j'}}\in A$ we define \textbf{index-distance} of  $u_{i_j}$ and $u_{i_{j'}}$ by $id(u_{i_j}, u_{i_{j'}})=min\{|i_j-i_{j'}|,\frac{n}{2}-|i_j-i_{j'}|\}$.
\end{def}

\begin{observation}\label{o22}\cite{mmnj}
Let $A=\{u_{i_1},u_{i_2},\cdots,u_{i_k}\}\subseteq U$ be a set such that $1\le i_1 <i_2<\cdots< i_k\le\frac{n}{2}$ and let $n_1, n_2, \cdots, n_k$ be the corresponding cyclic-sequence of $A$. Then,\\
\emph{(1)} $n_1+n_2+\cdots+n_k=\frac{n}{2}$.\\
\emph{(2)} If $u_{i_j}, u_{i_{j'}}\in A$, then $id(u_{i_j}, u_{i_{j'}})$ equals to sum of some consecutive elements of the cyclic-sequence of $A$ and $\frac{n}{2}-id(u_{i_j}, u_{i_{j'}})$ is sum of the remaining elements of the cyclic-sequence. Furthermore, $\{id(u_{i_j},u_{i_{j'}}),\frac{n}{2}-id(u_{i_j},u_{i_{j'}})\}=\{|i_j-i_{j'}|,\frac{n}{2}-|i_j-i_{j'}|\}$.
\end{observation}

We henceforth use the notation $\mathscr{M}_{\Delta}=\{ 2^a-2^b:0 \leq b<a< \Delta \}$ for $\Delta\geq 2$.

\begin{lemma}\label{l23}\cite{mmnj}
In the Kn\"odel graph $W_{\Delta,n}$ with vertex set $U\cup V$, for two distinct vertices $u_i$ and $u_j$, $N(u_i)\cap N(u_j)\ne \emptyset$ if and only if $id(u_i,u_j)\in\mathscr{M}_{\Delta}$ or $\frac{n}{2}-id(u_i,u_j)\in\mathscr{M}_{\Delta}$.
\end{lemma}

\begin{lemma}\label{l24}\cite{mmnj}
Let $W_{\Delta,n}$ be a Kn\"odel graph with vertex set $U\cup V$. For any non-empty subset $A\subseteq U$, the corresponding cyclic-sequence of $A$ has at most $\Delta |A|-|N(A)|$ elements belonging to $\mathscr{M}_{\Delta}$.
\end{lemma}

We remark that one can define the cyclic-sequence and index-distance for any subset of $V$ in a similar way, and thus the Observation \ref{o22} and Lemmas \ref{l23} and \ref{l24} are valid for cyclic-sequence and index-distance on subsets of $V$ as well.

The following property of Kn\"odel graphs is useful.
\begin{lemma}\label{l25}
Let $W_{\Delta,n}$ be a Kn\"odel graph with vertex set $U\cup V$. We have $\gamma(W_{\Delta,n})-1\le\gamma(W_{\Delta,n}-w)\le\gamma(W_{\Delta,n})$ for any vertex $w\in U\cup V$.
\end{lemma}

\begin{proof}
By transitivity of $W_{\Delta,n}$, for any vertex $w\in U\cup V$, there exist a $\gamma$-set, namely $D$, of $W_{\Delta,n}$ such that $w\notin D$. It is obvious that $D$ is a dominating set of $W_{\Delta,n}-w$ and so $\gamma(W_{\Delta,n}-w)\le\gamma(W_{\Delta,n})$. Now, let $D_w$ be a $\gamma$-set of $W_{\Delta,n}-w$. Immediately, we see that $D_w\cup\{w\}$ is a dominating set of $W_{\Delta,n}$ and we have $\gamma(W_{\Delta,n})\le\gamma(W_{\Delta,n}-w)+1$.
\end{proof}


\section{\bf  3-regular Kn\"odel graphs}
In this section, we will determine which 3-regular Kn\"odel graphs are $\gamma$-critical or $\gamma$-stable. For this, by vertex transitivity of Kn\"odel graphs, we remove the vertex $v_1$ and then we compare the  domination numbers of $W_{3,n}-v_1$ and $W_{3,n}$.

\begin{theorem}\label{t31}
A $3$-regular Kn\"odel graph $W_{3,n}$ is $\gamma$-critical if and only if $n\equiv 4$ \emph{(mod 8)}.
\end{theorem}
\begin{proof}
We have four cases in terms of  $n$ respect to modulo $8$. We will show that for $n\equiv 4$ (mod 8), the graph $W_{3,n}$ is $\gamma$-critical and in the other three cases $W_{3,n}$ is $\gamma$-stable. By vertex transitivity of Kn\"odel graphs, we use the vertex $v_1$ as an arbitrary vertex.\\
First assume that $n\equiv 4$ (mod 8). We set $n=8t+4$, where $t$ is a positive integer. By Theorem \ref{t11}, we know that $\gamma(W_{3,8t+4})=2t+2$. On the other hand, the set $D=\{u_{4i-2}:i=1,2,\cdots,t\}\cup\{v_{4i}:i=1,2,\cdots,t\}\cup\{v_{4t+2}\}$ is a dominating set for $W_{3,8t+4}-v_1$ and so $\gamma(W_{3,8t+4}-v_1)\le 2t+1<\gamma(W_{3,8t+4})$. Therefore $W_{3,8t+4}$ is $\gamma$-critical and by Observation \ref{o14}, $\gamma(W_{3,8t+4}-v_1)= 2t+1$.\\
For the rest of cases we use Theorems \ref{t11} and \ref{t13}.\\
If $n\equiv 0$ (mod 8), then we set $n=8t$, where $t$ is a positive integer. By Theorem \ref{t11}, we know that $\gamma(W_{3,8t})=2t$ and by Theorem \ref{t13} we have $\gamma(W_{3,8t}-v_1)\ge \lceil\frac{8t-1}{4}\rceil=2t= \gamma(W_{3,8t})$ and so $W_{3,8t}$ is not $\gamma$-critical. If $n\equiv 2$ (mod 8), then we set $n=8t+2$, where $t$ is a positive integer. By Theorem \ref{t11}, we know that $\gamma(W_{3,8t+2})=2t+1$ and by Theorem  \ref{t13}, we have $\gamma(W_{3,8t+2}-v_1)\ge \lceil\frac{8t+1}{4}\rceil=2t+1=\gamma(W_{3,8t+2})$ and so $W_{3,8t+2}$ is not $\gamma$-critical. If $n\equiv 6$ (mod 8), then we set $n=8t+6$, where $t$ is a positive integer. By Theorem  \ref{t11}, we know that $\gamma(W_{3,8t+6})=2t+2$ and by Theorem \ref{t13}, we have $\gamma(W_{3,8t+6}-v_1)\ge \lceil\frac{8t+5}{4}\rceil=2t+2=\gamma(W_{3,8t+6})$ and so $W_{3,8t+6}$ is not $\gamma$-critical.
 \end{proof}
Theorem \ref{t31} and Lemma \ref{l25} implies that the following corollary.
\begin{corollary}
A $3$-regular Kn\"odel graph $W_{3,n}$ is $\gamma$-stable if and only if $n\not\equiv 4$ \emph{(mod 8)}.
\end{corollary}


\section{\bf 4-regular Kn\"odel graphs}

In this section, we discuss on $\gamma$-critically of 4-regular Kn\"odel graphs. We study our main result in four lemmas.

\begin{lemma}\label{l41}
$W_{4,n}$ is $\gamma$-critical if \\
\emph{i)} $n=26$  or\\
\emph{ii)} $n\ge22$ and $n\equiv 2$ \emph{(mod 10)}  or\\
\emph{iii)} $n\ge38$ and $n\equiv 8$ \emph{(mod 10)}.
\end{lemma}
\begin{proof}
We prove the $\gamma$-criticality of $W_{4,n}$ in three given cases for $n$. In each case, we present a dominating set for $W_{4,n}-v_1$ with $\gamma(W_{4,n})-1$ vertices that implies $W_{4,n}$  is  $\gamma$-critical.\\
(i) If $n=26$, then the set $D=\{u_2,u_{10}\}\cup \{v_6,v_7,v_8,v_{12}\}$ is a dominating set for $W_{4,26}-v_1$ with 6 vertices and so $\gamma(W_{4,26}-v_1)\le 6$. By Theorem \ref{t12} we know that $\gamma(W_{4,28})=7$. Therefore, by Lemma \ref{l25} we have $\gamma(W_{4,28}-v_1)=6<\gamma(W_{4,28})$ and so $W_{4,28}$  is  $\gamma$-critical.\\
(ii) If $n\ge22$ and $n\equiv 2$  (mod 10), then we set $n=10t+2$, where $t\ge2$. The set  $D=\{u_{5i-1}:i=1,2,\cdots,t-1\}\cup\{u_{5t}\}\cup \{v_{5i-2}:i=1,2,\cdots,t\}\cup\{v_{5t-1}\}$ is a dominating set for $W_{4,10t+2}-v_1$ with $2t+1$ vertices and so $\gamma(W_{4,10t+2}-v_1)\le 2t+1$. By Theorem \ref{t12} we know that $\gamma(W_{4,10t+2})=2t+2$. Therefore, by Lemma \ref{l25} we have $\gamma(W_{4,10t+2}-v_1)=2t+1<\gamma(W_{4,10t+2})$ and so $W_{4,10t+2}$  is  $\gamma$-critical, where $t\ge 2$.\\
(iii) If $n\ge38$ and $n\equiv 8$  (mod 10), then we set $n=10t+8$, where $t\ge3$. The set  $D=\{u_{5i-1}:i=1,2,\cdots,t\}\cup\{u_{5t}\}\cup \{v_{5i-2}:i=1,2,\cdots,t+1\}\cup\{v_6,v_{5t-1}\}$ is a dominating set for $W_{4,10t+8}-v_1$ with $2t+3$ vertices and so $\gamma(W_{4,10t+8}-v_1)\le 2t+3$. By Theorem \ref{t12} we know that $\gamma(W_{4,10t+8})=2t+4$. Therefore, by Lemma \ref{l25} we have $\gamma(W_{4,10t+8}-v_1)=2t+3<\gamma(W_{4,10t+8})$ and so $W_{4,10t+8}$  is  $\gamma$-critical, where $t\ge 3$.
\end{proof}

Now, we will prove that the Kn\"odel graphs $W_{4,n}$ are $\gamma$-stable for all positive integer $n$ where $n$ does not satisfy in Lemma \ref{l41}. In each cases, we show that $\gamma(W_{4,n}-w)=\gamma(W_{4,n})$ for any vertex $w\in U\cup V$.

\begin{lemma}\label{l42}
The Kn\"odel graph $W_{4,n}$ is $\gamma$-stable if $n\equiv 0$ \emph{(mod 10)}  or $n\equiv 4$ \emph{(mod 10)}.
\end{lemma}

\begin{proof}
If $n\equiv 0$ (mod 10), then we set $n=10t$, where $t\ge2$. By Theorem \ref{t12}, we know that $\gamma(W_{4,10t})=2t$. Also, by Theorem \ref{t13}, $\gamma(W_{4,10t}-w)\ge \lceil \frac{10t-1}{5}\rceil=2t$ for any vertex $w\in U\cup V$ and therefore, by Lemma \ref{l25} we have $\gamma(W_{4,10t}-w)=\gamma(W_{4,10t})=2t$. Thus, $W_{4,10t}$ is $\gamma$-stable, where $t\ge2$.\\
If $n\equiv 4$ (mod 10), then we set $n=10t+4$, where $t\ge2$. By Theorem \ref{t12}, we know that $\gamma(W_{4,10t+4})=2t+2$. Also, by Theorem \ref{t13}, we have $\gamma(W_{4,10t+4}-w)\ge \lceil \frac{10t+3}{5}\rceil=2t+1$ for any vertex $w\in U\cup V$. We will show that $\gamma(W_{4,10t+4}-w)\ne 2t+1$. Without loss of generality, we assume that $w\in V$.
Suppose to the contrary, $D$ is a dominating set of $W_{4,2t+4}-w$ and $|D|=2t+1$. If $x=|D_U|$ and $y=|D\cap(V\setminus\{w\})|$, then $x+y=2t+1$. Since $D$ is a dominating set, we have $4x+y=4x+(2t+1-x)\ge5t+1=|V\setminus\{w\}|$ and $4y+x=4y+(2t+1-y)\ge5t+2=|U|$, therefore $x\ge t$ and $y\ge t+1$. Now, we deduce that $x=t$ and $y=t+1$. Assume that $n_1,n_2,\cdots, n_t$ is the cyclic-sequence of $D_U$. Thus, by Observation \ref{o22}, $n_1+n_2+\cdots+n_t=5t+2$ and so $\{n_1,n_2,\cdots, n_t\}\cap \mathscr{M}_4\ne\emptyset$. Now, Lemma \ref{l23} implies that $D_U$ dominates at most $4t-1$ vertices of $V\setminus\{w\}$ and $D$ dominates at most $4t-1+t+1=5t$ vertices of $V\setminus\{w\}$, a contradiction. Hence, $\gamma(W_{4,10t+4}-w)\ne 2t+1$, as desired. Therefore, $\gamma(W_{4,10t+4}-w)=2t+2=\gamma(W_{4,10t+4})$, and so $W_{4,10t+4}$ is $\gamma$-stable, where $t\ge2$.
\end{proof}

\begin{lemma}\label{l43}
The Kn\"odel graph $W_{4,n}$ is $\gamma$-stable if $n\in\{16,18,28,36\}$.
\end{lemma}
\begin{proof}
If $n=16$, then by Theorem \ref{t12} we have $\gamma(W_{4,16})=4$. We show that $\gamma(W_{4,16}-w)=4$ for any vertex $w\in U\cup V$. Without loss of generality, we assume that $w\in V$. Suppose to the contrary that $\gamma(W_{4,16}-w)\le3$ and $D$ is a dominating set of $W_{4,16}-w$ and $|D|=3=x+y$, where $x=|D_U|$, $y=|D\cap(V\setminus\{w\})|$. $D$ dominates at most $4x+y=3x+3$ vertices of $V\setminus\{w\}$ and dominates at most $4y+x=3y+3$ vertices of $U$. Since $D$ is a dominating set, we have $3x+3\ge 7$ and $3y+3\ge8$ and so $x\ge2$, $y\ge2$ and $x+y\ge4$, a contradiction. Hence, by Lemma \ref{l25} we have $\gamma(W_{4,16}-w)=\gamma(W_{4,16})=4$ and $W_{4,16}$ is $\gamma$-stable.\\
If $n=36$, then by Theorem \ref{t12} we have $\gamma(W_{4,36})=8$. We show that $\gamma(W_{4,36}-w)=8$. for any vertex $w\in U\cup V$. Without loss of generality, we assume that $w\in V$. Suppose to the contrary that $\gamma(W_{4,36}-w)\le7$ and $D$ is a dominating set of $W_{4,36}-w$ and $|D|=7=x+y$, where $x=|D_U|$, $y=|D\cap(V\setminus\{w\})|$. $D$ dominates at most $4x+y=3x+7$ vertices of $V\setminus\{w\}$ and dominates at most $4y+x=3y+7$ vertices of $U$. Since $D$ is a dominating set, we have $3x+7\ge18$ and $3y+3\ge17$ and so $x\ge4$, $y\ge4$ and $x+y\ge8$, a contradiction. Hence, by Lemma \ref{l25} we have $\gamma(W_{4,36}-w)=\gamma(W_{4,36})=8$ and $W_{4,36}$ is $\gamma$-stable.\\
If $n=18$, then by Theorem \ref{t12} we have $\gamma(W_{4,18})=4$. We show that $\gamma(W_{4,18}-w)=4$ for any vertex $w\in U\cup V$. Without loss of generality, we assume that $w\in V$. Suppose to the contrary that $\gamma(W_{4,18}-w)\le 3$ and $D$ is a dominating set of $W_{4,18}-w$ and $|D|=3=x+y$, where $x=|D_U|$, $y=|D\cap(V\setminus\{w\})|$. $D$ dominates at most $4x+y=3x+3$ vertices of $V\setminus\{w\}$ and dominates at most $4y+x=3y+3$ vertices of $U$. Since $D$ is a dominating set, we have $3x+3\ge 8$ and $3y+3\ge 9$ and so $x\ge 2$, $y\ge 2$ and $x+y\ge 4$, a contradiction. Hence, by Lemma \ref{l25} we have $\gamma(W_{4,18}-w)=\gamma(W_{4,18})=4$ and $W_{4,18}$ is $\gamma$-stable.\\
If $n=28$, then by Theorem \ref{t12} we have $\gamma(W_{4,28})=7$. We show that $\gamma(W_{4,28}-w)=7$ for any vertex $w\in U\cup V$. Without loss of generality, we assume that $w\in V$. Suppose to the contrary that $\gamma(W_{4,28}-w)\le 6$ and $D$ is a dominating set of $W_{4,28}-w$ and $|D|=6=x+y$, where $x=|D_U|$, $y=|D\cap(V\setminus\{w\})|$. $D$ dominates at most $4x+y=3x+3$ vertices of $V\setminus\{w\}$ and dominates at most $4y+x=3y+3$ vertices of $U$. Since $D$ is a dominating set, we have $3x+3\ge 13$ and $3y+3\ge 14$ and so $x\ge 3$, $y\ge 3$. Therefore, we have $x=|D_U|=3$ and $y=|D\cap(V\setminus\{w\})|=3$. Let $D\cap(V\setminus\{w\})=\{v_i,v_j,v_k\}$, where $1\le i<j<k\le14$. The corresponding cyclic-sequence of this set is $n_1=j-i$, $n_2=k-j$ and $n_3=14-k+i$. Since $n_1+n_2+n_3=14$, then $\{n_1,n_2,n_3\}\cap \mathscr{M}_4\ne\emptyset$ and by Lemma \ref{l23} we have $|N(\{v_i,v_j,v_k\})|\le11$. If $|N(\{v_i,v_j,v_k\})|\le10$, then $D$ dominates at most $10+3=13$ vertices of $U$, a contradiction. Thus, we must have $|N(\{v_i,v_j,v_k\})|=11$. In this case, Lemma \ref{l24} implies that precisely one of the numbers $n_1, n_2$ and $n_3$ belongs to $\mathscr{M}_4$. By symmetry, we have the only case $n_1=4$ and $n_2=n_3=5$ for cyclic-sequence of $\{v_i,v_j,v_k\}$. We can easily see that for any selection of $D_V$, $D$ dominates at most 26 vertices of $W_{4,28}-w$, a contradiction. For example, if $w\notin\{v_1,v_5,v_{10}\}$, then we consider the set $D_V=\{v_1,v_5,v_{10}\}$ and so $D_U=U\setminus N(D_V)=\{u_6,u_{11},u_{13}\}$. Now, $D=D_U\cup D_V$ does not dominate the vertices $v_3,v_8$, a contradiction.
\end{proof}

\begin{lemma}\label{l44}
$W_{4,n}$ is $\gamma$-stable if $n\ge46$ and $n\equiv 6$ .
\end{lemma}
\begin{proof}

If $n\ge46$ and $n\equiv 6$  (mod 10), then we set $n=10t+6$, where $t\ge4$. By Theorem \ref{t12}, we have $\gamma(W_{4,n})=2t+3$ and we show that $\gamma(W_{4,n}-w)=2t+3$ for any vertex $w\in U\cup V$.  Without loss of generality, we assume that $w\in V$. Suppose to the contrary that $\gamma(W_{4,n}-w)\le 2t+2$ and $D$ is a dominating set of $W_{4,n}-w$ and $|D|=2t+2=x+y$, where $x=|D_U|$ and $y=|D\cap(V\setminus\{w\})|$. $D$ dominates at most $4x+y=3x+2t+2$ vertices of $V\setminus\{w\}$ and dominates at most $4y+x=3y+2t+2$ vertices of $U$. Since $D$ is a dominating set, we have $3x+2t+2\ge 5t+2$ and $3y+2t+2\ge 5t+3$ and so $x\ge t$, $y\ge t+1$. Now, we have two cases (i) $x=t$, $y=t+2$ and (ii) $x=y=t+1$.

 \textbf{i)} Assume that $x=t$, $y=t+2$ and $n_1,n_2,\cdots,n_t$ is the cyclic-sequence of $D_U$ and so by Observation \ref{o22} we have $n_1+n_2+\cdots+n_t=5t+3$. Since $D_U$ has to dominate precisely $4t=|(V\setminus\{w\})\setminus(D_V)|$ vertices of $V$, Lemmas \ref{l23} and \ref{l24} implies that  $\{n_1,n_2,\cdots,n_t\}\cap\mathscr{M}_4=\emptyset$. Thus, there exist the unique cyclic-sequence $n_1=n_2=\dots= n_{t-1}=5$ and $n_t=8$, with respect to vertex transitivity property of Kn\"odel graphs. If $D_U=\{u_{5i-4}:i=1,2,\cdots,t\}$, then $ V\setminus N(D_U)=\{v_3,v_{5t+1},v_{5t+2}\}\cup\{v_{5i}:i=1,2,\cdots,t\}$ and $(V\setminus N(D_U))\subseteq D_V$. Now, we have $|D_V|=t+2$ and $|V\setminus N(D_U)|=t+3$ and therefore, $w\in (V\setminus N(D_U))$ and $D_V=(V\setminus N(D_U))\setminus\{w\}$. On the other hand, each vertex in $V\setminus N(D_U)$, including $w$, has at least an external private neighbor respect to $U$ as follows: $pn(v_{3})=\{u_{5t+3}\}$, $pn(v_{5})=\{u_4,u_{5}\}$,  $pn(v_{5t+1})=\{u_{5t-2}\}$, $pn(v_{5t+2})=\{u_{5t+2}\}$ and $u_{5i-3}\in pn(v_{5i})$, where $i=2,3,\cdots,t$. Therefore, by eliminating $w$, $D$ is not a dominating set of $W_{4,10t+6}-w$, a contradiction and so $x\ne t$.

\textbf{ii)} Assume that $x=t+1$, $y=t+1$ and $n_1,n_2,\cdots,n_{t+1}$ is the cyclic-sequence of $D_V$ and so by Observation \ref{o22} we have $n_1+n_2+\cdots+n_{t+1}=5t+3$. Since $D_V$ has to dominate at least $4t+2=|U\setminus D_U|$ vertices of $U$, thus there exist at most two pairs of vertices in $U$ that are dominated by at least two vertices in $D_V$. Therefore, by lemma \ref{l23}, up to two of numbers
  $n_1,n_2, \cdots, n_{t+1}, n_1+n_2,n_2+n_3, \cdots, n_{t}+n_{t+1}, n_{t+1}+n_1$ are belong to $\mathscr{M}_4$. We divide all the sequences that apply in this condition into 6 categories. In each category, we consider the set $D_V$ corresponding to the given cyclic-sequence such that $w\notin D_V$. Next, we show that $|N[D]|\le 10t+4$, that is, the constructed $D$ is not a dominating set, a contradiction.We emphasize that $(V\setminus N(D_V))\subseteq D_U$ and $D=D_U\cup D_V$.
  
 \textbf{1)} $n_1=3, n_2=n_3=\dots=n_{t+1}=5$.\\
 If $D_V=\{v_1\}\cup\{v_{5i-1}:i=1,2,\cdots, t\}$, then $U\setminus N(D_V)=\{u_{5i}:i=1,2,\cdots, t-1\}\cup\{u_{5t+2}\}$. Until now, we have found $2t+1$ vertices of the set $D$ and we are allowed to add just one more vertex to $D_U$. The known vertices of $D$ does not dominate the vertices $v_3, v_7, v_{5t}$ and $v_{5t+1}$ and by adding each new vertex to $D$, there are still 2 vertices that are not dominated. Therefore, $D$ dominates at most $10t+4$ vertices of $W_{4,n}-w$, a contradiction.
 
\textbf{2)} $n_1=n_r=4$, for an $r$, $2\le r\le t+1 $ and other $t-1$ elements of cyclic-sequence are equal to $5$.\\
First, assume that $r=2$, then we have $n_1=n_2=4$ and $n_3=n_4=\cdots=n_{t+1}=5$. If $D_V=\{v_1, v_5\}\cup\{v_{5i-1}:i=2,3,\cdots, t\}$, then $D_U=U\setminus N(D_V)=\{u_3, u_{5t+2}\}\cup\{u_{5i}:i=2,3,\cdots, t\}$. $D=D_U\cup D_V$ does not dominate the vertices $v_7, v_8$ and $v_{12}$, a contradiction. 
Now, assume that $3\le r\le t+1$, then we have $v_3\notin D_V$ and $N(v_3)=\{u_2, u_3, u_{5t+3}, u_{5t-1}\}$. We show that $N(v_3)\subseteq N(D_V)$. \\
$u_2\in N(v_5)$ and $v_5\in D_V$ and thus $u_2\in N(D_V)$,
$u_3\in N(v_{10})$ and $v_{10}\in D_V$ and thus $u_3\in N(D_V)$,
$u_{5t+3}\in N(v_1)$ and $v_1\in D_V$ and thus $u_{5t+3}\in N(D_V)$, $u_{5t-1}\in N(v_{5t-1})$ and $u_{5t-1}\in N(v_{5t-1})$. If $r\ne t+1$, then $v_{5t-1}\in D_V$ and if $r=t+1$, then $v_{5t}\in D_V$. In both cases we see that $u_{5t-1}\in N(D_V)$.\\
Therefore, $N(v_3)\subseteq N(D_V)$ and so $N(v_3)\cap D_U=\emptyset$. Since $v_3\notin D_V$, thus $N[v_3]\cap D=\emptyset$ or $v_3\notin N[D]$, a contradiction.

\textbf{3)} $n_1=4,n_2=1,n_3=8$ and $n_4=n_5=\cdots=n_{t+1}=5$.\\
If $D_V=\{v_1,v_5,v_6\}\cup\{v_{5i-1}:i=3,4,\cdots,t\}$, then $D_U=U\setminus N(D_V)=\{u_8,u_9\}\cup\{u_{5i}:i=2,4,\cdots,t\}$. $D=D_U\cup D_V$ does not dominate three vertices $v_2, v_3$ and $v_7$, a contradiction.

\textbf{4)} $n_1=8,n_2=1,n_3=4$ and $n_4=n_5=\cdots=n_{t+1}=5$.\\
If $D_V=\{v_1,v_9,v_{10}\}\cup\{v_{5i-1}:i=3,4,\cdots,t\}$, then $D_U=U\setminus N(D_V)=\{u_4,u_5,u_{5t+2}\}\cup\{u_{5i}:i=3,4,\cdots,t\}$. $D=D_U\cup D_V$ does not dominate three vertices $v_3,v_{13}$ and $v_{17}$, a contradiction.

\textbf{5)} $n_1=3,n_2=2,n_3=8$ and $n_4=n_5=\cdots=n_{t+1}=5$.\\
If $D_V=\{v_1,v_4,v_6\}\cup\{v_{5i-1}:i=3,4,\cdots,t\}$, then $D_U=U\setminus N(D_V)=\{u_2,u_8,u_9\}\cup\{u_{5i-5}:i=3,4,\cdots,t\}$. $D=D_U\cup D_V$ does not dominate four vertices $v_7,v_{5t},v_{5t+1}$ and $v_{5t+3}$, a contradiction.

\textbf{6)} $n_1=8,n_2=2,n_3=3$ and $n_4=n_5=\cdots=n_{t+1}=5$.\\
If $D_V=\{v_1,v_9,v_{11}\}\cup\{v_{5i-1}:i=3,4,\cdots,t\}$, then $D_U=U\setminus N(D_V)=\{u_3,u_5,u_{5t+2}\}\cup\{u_{5i}:i=3,4,\cdots,t\}$. $D=D_U\cup D_V$ does not dominate two vertices $v_7,v_{13}$, a contradiction.
\end{proof}
Finally, by lemmas \ref{l41}, \ref{l42}, \ref{l43} and \ref{l44}, we have the main results as follows.

\begin{theorem}
For any $n\ge16$, $W_{4,n}$ is $\gamma$-critical if and only if\\
\emph{i)} $n=26$  or\\
\emph{ii)} $n\ge22$ and $n\equiv 2$ \emph{(mod 10)}  or\\
\emph{iii)} $n\ge38$ and $n\equiv 8$ \emph{(mod 10)}.
\end{theorem}

\begin{theorem}
For any $n\ge16$, $W_{4,n}$ is $\gamma$-stable if and only if\\
\emph{i)} $n=18,28$  or\\
\emph{ii)} $n\equiv 0,4$ \emph{(mod 10)}  or\\
\emph{iii)} $n\ne26$ and $n\equiv 6$ \emph{(mod 10)}.
\end{theorem}


\end{document}